\documentclass[12pt]{amsart}

\usepackage{amssymb}
\usepackage{latexsym}

\usepackage{exscale}
\usepackage{lscape}

\usepackage[colorlinks=true, pdfstartview=FitV, linkcolor=blue, citecolor=blue, urlcolor=blue]{hyperref}

\newtheorem{theorem}{Theorem}[section]
\newtheorem{lemma}[theorem]{\bf Lemma}
\newtheorem{cor}[theorem]{Corollary}

\newtheorem{example}[theorem]{Example}

\newtheorem{remark}[theorem]{\textbf{Remark}}

\allowdisplaybreaks

\numberwithin{equation}{section}

\newcommand{\R}{\mathbb R}

\newcommand{\C}{\mathbb{C}}

\newcommand{\D}{\partial}
\newcommand{\Dbar}{\bar{\partial}}
\newcommand{\al}{\alpha}

\newcommand{\ep}{\varepsilon}

\DeclareMathOperator{\supp}{supp}

\DeclareMathOperator{\diam}{diam}

\DeclareMathOperator{\re}{Re}
\DeclareMathOperator{\im}{Im}

\newcommand{\grad}{\nabla}
\DeclareMathOperator{\Div}{div}

\def\Xint#1{\mathchoice
   {\XXint\displaystyle\textstyle{#1}}%
   {\XXint\textstyle\scriptstyle{#1}}%
   {\XXint\scriptstyle\scriptscriptstyle{#1}}%
   {\XXint\scriptscriptstyle\scriptscriptstyle{#1}}%
   \!\int}
\def\XXint#1#2#3{{\setbox0=\hbox{$#1{#2#3}{\int}$}
     \vcenter{\hbox{$#2#3$}}\kern-.5\wd0}}

\def\avgint{\Xint-}

\begin{document}

\title[Regularity results for elliptic PDEs]{Regularity results for weak
  solutions of elliptic PDEs below the natural exponent}

\author{David Cruz-Uribe, SFO}
\address{David Cruz-Uribe, SFO\\
Dept. of Mathematics \\ Trinity College \\
Hartford, CT 06106-3100, USA} 
\email{david.cruzuribe@trincoll.edu}

\author{Kabe Moen}
\address{Kabe Moen \\ 
Department of Mathematics \\
 University of Alabama \\
 Tuscaloosa, AL 35487, USA}
\email{kabe.moen@ua.edu}

\author{Scott Rodney}
\address{Scott Rodney \\
Department of Mathematics, Physics and Geology \\
 Cape Breton University\\
Sydney, NS B1P6L2, Canada}

\email{scott\_rodney@cbu.ca}

% \author{Carlo Sbordone}
% \address{Dipartimento di Matematica e Applicazioni
%   ``R. Caccioppoli'' Universit\`a \\
% Via Cintia, 80126 Napoli, Italy}

% \email{sbordone@unina.it}

\thanks{The first author is supported by the Stewart-Dowart faculty
  development fund at Trinity College and NSF Grant 1362425.  The second author is supported
  by NSF Grant 1201504.  The third author is supported by the NSERC
  Discovery Grant program.  The authors would like to thank Carlo
  Sbordone for originally suggesting this problem and for his
  hospitality to the first author.  They would also like to thank
  Doyoon Kim for pointing out a mistake in an earlier version of this article.}



\date{November 4, 2014}
\begin{abstract}
We prove regularity estimates for the weak solutions to the
Dirichlet problem for a divergence form elliptic operator.  We give
$L^p$ estimates for the second derivative for $p<2$.  Our work
generalizes results due to Miranda~\cite{MR0170090}.
\end{abstract}

\maketitle

\section{Introduction}

In this paper we consider the regularity of solutions to the
divergence form elliptic equation
\begin{equation} \label{eqn:main1}
\begin{cases}
Lu = -\Div A\grad u = f & \text{in } \Omega\\
\;\;u = 0  &  \text{on } \partial \Omega
\end{cases}
\end{equation}
where $\Omega \subset \R^n$, $n\geq 2$, is a bounded open set whose
boundary $\partial \Omega$ is $C^2$, and $A=A(x)=(a_{ij}(x))$ is an $n\times
n$  matrix of real-valued, measurable functions that satisfies the ellipticity
condition
\begin{equation}
 \label{E1} \lambda |\xi|^2 \leq\langle A\xi,\xi \rangle \leq \Lambda |\xi|^2, 
\qquad  0< \lambda < \Lambda, \quad \xi \in \R^n. 
\end{equation}
We derive $L^p$ estimates, $p<2$, for solutions of this equation when $A$ has
discontinuous coefficients and $f\in L^p(\Omega)$.

This and
related problems have a long history.  If $A$ is continuous and
$\partial\Omega$ is $C^{2,\alpha}$, then
these results are classical:  see Gilbarg and
Trudinger~\cite{gilbarg-trudinger01}.   Miranda~\cite{MR0170090} showed
that if $n\geq 3$, $\partial \Omega$ is $C^3$, and $A\in W^{1,n}(\Omega)$, then any weak solution of
$L u = f$, $f\in L^q(\Omega)$, $q\geq 2$, is a strong solution and
$\|D^2u \|_{L^2(\Omega)} \leq C(\|f\|_{L^q(\Omega)}+\|u\|_{L^1(\Omega)})$.
This result is false when $n=2$: for a counter-example, see Example~\ref{example:n=2}
below.

A similar problem for non-divergence form elliptic operators was
considered by Chiarenza and Franciosi~\cite{MR1174821}.  They proved that if
$n\geq 3$, $\Omega$ is bounded and $\partial \Omega$ is $C^{2}$, then the
non-divergence form equation $\mathrm{tr}(A D^2u) =f$, with $f\in L^2(\Omega)$ and
$A$ in a certain vanishing Morrey class (a generalization of
$VMO$), has a unique solution $u$ satisfying $\|u\|_{W^{2,2}(\Omega)} \leq
C\|f\|_{L^2(\Omega)}$.  This was generalized by Chiarenza, Frasca and
Longo~\cite{MR1088476}, who showed that if $f\in L^p$, $1<p<\infty$,
then the same equation has a unique solution satisfying $\|u\|_{W^{2,p}(\Omega)} \leq
C\|f\|_{L^p(\Omega)}$.  These results in turn were further generalized by
Vitanza~\cite{MR1229461,MR1320669,MR1283360}.

Divergence form equations of the form $\Div A\grad u = \Div F$ were considered by 
Di~Fazio~\cite{MR1405255} on bounded domains with  $\partial \Omega
\in C^{1,1}$ and Iwaniec and Sbordone~\cite{MR1631658} on $\R^n$; they showed that if $A\in VMO$, then 
there exists a unique weak solution that  satisfies
$\|\grad u \|_{L^p(\Omega)} \leq C\|F\|_{L^p(\Omega)}$,  $1<p<\infty$.
 The results for bounded domains were improved by Auscher and
Qafsaoui~\cite{MR1911202}, who showed that it suffices to assume
$\partial \Omega$ is~$C^1$ and that $A$ does not need to be real
symmetric.  For a generalization to nonlinear equations,
see~\cite{MR1489454}.   N.~Meyers considered the more general equation
$\Div A\grad u = \Div F+f$ on a bounded domain with a smooth
boundary.  He showed that if $A$ satisfies~\eqref{E1}, then there
exists $p_0<2$ such that for all $p_0<p<p_0'$, there exists a weak
solution that satisfies $\|\grad u \|_{L^p(\Omega)} \leq
C(\|F\|_{L^p(\Omega)}+\|f\|_{L^p(\Omega)})$.  (See
Theorem~\ref{thm:Meyers} below.)

\bigskip

Our main result is a generalization of the result of Miranda to
$p< 2$ and $n\geq 2$.    

\begin{theorem} \label{thm:main-result}
Let $\Omega \subset \R^n$,
  $n\geq 2$, be a bounded open set such that $\partial \Omega$ is $C^1$.
Let  $A$ be an $n\times n$ real-valued matrix that
satisfies~\eqref{E1}.  If $A\in
  W^{1,n}(\Omega)$, then there exists $p_0\in(1,2)$ so that for all
  $p\in(p_0,2)$ and $f\in L^p(\Omega)$ there exists a unique solution $u$ of
  \eqref{eqn:main1} that satisfies a local regularity estimate:
 given any open set $\Omega '\Subset \Omega$, 
\begin{equation} \label{eqn:mainNEW} 
\|D^2 u \|_{L^p(\Omega')} \leq D^{-1}C\|f\|_{L^p(\Omega)},
\end{equation}
where $C$ is independent of both $u$ and $f$ and $D=d(\Omega',\partial\Omega)$. 
If we further assume that $\partial \Omega$ is $C^2$, then 
\begin{equation} \label{eqn:main2}
 \|D^2 u \|_{L^p(\Omega)} \leq C\|f\|_{L^p(\Omega)}.
\end{equation}
where $C$ is independent of both $u$ and $f$.
\end{theorem}

\begin{remark}
 To compare Theorem~\textup{\ref{thm:main-result}} to the work of Di Fazio
  {\em et al.} described above, note that if $A\in W^{1,n}$ then $A\in
  VMO$: see, for instance,~\cite{MR1650446}.
\end{remark}

\begin{remark} 
Our techniques actually allow us to assume that $A$ is a complex matrix that satisfies
\[  |\langle A\xi,\eta\rangle|\leq \Lambda |\xi||\eta|, 
\qquad \lambda|\xi|^2\leq\mathrm{Re}\langle A\xi,\xi\rangle, \quad
\xi,\eta \in \C^n. \]
Details are left to the interested reader.
\end{remark}

\medskip

The lower bound $p_0$ in Theorem~\ref{thm:main-result} is intrinsic to
our method of proof.   It is an open question
whether our results can be extended to the full range $1<p<2$.  The stronger
assumptions on the boundary to get global regularity in Theorem~\ref{thm:main-result} are not
unexpected:  there exist examples that show that for $n\geq 2$ and $p>1$, there
exists a bounded $C^1$ domain $\Omega$ and $f\in
C^\infty(\overline{\Omega})$ such that the solution $u$ to $\Delta u =
f$ in $\Omega$, $u=0$ on $\partial \Omega$, is not in
$W^{2,1}(\Omega)$.  (See~\cite{MR0544584,MR1331981}.)

\medskip

When $n\geq 3$, an examination of the constants shows that we can
take $p=2$ in our proof.    This lets us give a new proof of the result of
Miranda referred to above, as well as a local regularity result.

\begin{cor} \label{cor:miranda}
Let $\Omega \subset \R^n$,
  $n\geq 3$, be a bounded open set such that $\partial \Omega$ is $C^1$.
Let  $A$ be an $n\times n$ real-valued matrix that
satisfies~\eqref{E1}.  If $A\in
  W^{1,n}(\Omega)$, then for all
 $f\in L^2(\Omega)$, there exists a unique solution $u$ of
  \eqref{eqn:main1} that satisfies
\[
 \|D^2 u \|_{L^2(\Omega')} \leq D^{-1}C\|f\|_{L^2(\Omega)},
\]
where $C$ is independent of both $u$ and $f$, $\Omega ' \Subset
\Omega$ and $D=d(\Omega',\partial\Omega)$.  If we further assume that
$\partial \Omega$ is $C^2$, then 
\[
 \|D^2 u \|_{L^2(\Omega)} \leq C\|f\|_{L^2(\Omega)}.
\]
\end{cor}

\medskip

We now consider the case $p=n=2$.  In this case,  Corollary~\ref{cor:miranda} is false, as the next
example shows.  

\begin{example} \label{example:n=2} 
  Let $B=B_{1/2}(0)\subset \R^2$ be the open ball of radius $1/2$
  centered at the origin.  Then there exists a matrix $A\in W^{1,2}(B)$ satisfying
  \eqref{E1} and a solution to 
\[ -\Div(A\nabla u)=0 \]
such that $u\in W^{2,p}(B)$ for all $p<2$, but $u\notin W^{2,2}(B)$.
\end{example}

We can adapt the proofs of Theorem~\ref{thm:main-result} to the case $p=n=2$ if we  assume that
$\grad A$ satisfies stronger integrability conditions.   We
state these in the scale of Orlicz spaces---for a precise definition, see
Section~\ref{section:prelim} below.   For brevity we only state the
global regularity result.

\begin{theorem} \label{thm:n=2-special}
Let $\Omega \subset \R^2$ be a bounded open set such that $\partial \Omega$ is $C^2$.
Let  $A$ be a $2\times 2$ real-valued matrix that
satisfies~\eqref{E1}.  Suppose further that for some $\delta>0$,
\begin{equation} \label{eqn:LlogL-hyp}
\|\nabla A\|_{L^2(\log L)^{1+\delta}(\Omega)}<\infty.
\end{equation}
If $f\in L^2(\Omega)$ then there exists a unique solution $u$ of
\eqref{eqn:main1} that 
satisfies
$$\|D^2u\|_{L^2(\Omega)}\leq C \|\nabla A\|_{L^2(\log L)^{1+\delta}(\Omega)} \|f\|_{L^2(\Omega)}.$$
\end{theorem}

Our second result gives information in the end point case when
$\delta=0$.   In this case we need to impose an additional regularity
condition.   Recall (cf.~\cite{MR861756}) that if $\Omega\subset \R^2$, a function $u$ is contained in the Morrey
space $L^{2,\lambda}(\Omega)$ if 
\[ \|u\|_{L^{2,\lambda}(\Omega)} = \sup_{Q}  \left(|Q|^{-\frac{\lambda}{2}} \int_{Q\cap \Omega}
  u^2\,dx\right)^{1/2} < \infty. \]

\begin{theorem} \label{thm:n=2-off-diag}
Let $\Omega \subset \R^2$ be a bounded open set such that $\partial \Omega$ is $C^2$.
Let  $A$ be a $2\times 2$ real-valued matrix that
satisfies~\eqref{E1} and 
\[ \|A\|_{L^2(\log L)(\Omega)}<\infty. \]
Suppose further that for some
$1<r<2$,
$\nabla A \in L^{2,\frac{4}{r}-2}(\Omega)$.
If $f\in L^2(\Omega)$ then there exists a unique solution $u$ of \eqref{eqn:main1}
that satisfies
$$\|D^2u\|_{L^2(\Omega)}\leq C(r,\Omega) \|\nabla
A\|_{L^{2,\frac{4}{r}-2}(\Omega)} \|f\|_{L^2(\Omega)}.$$
\end{theorem}

Unfortunately, both of these results are weaker than they appear.  In
two dimensions, \eqref{eqn:LlogL-hyp} implies that $\nabla A$ is
continuous: see Cianchi~\cite{MR2265898,MR2723814}.  Similarly, if we
assume that $\nabla A \in L^{2,\frac{4}{r}-2}(\Omega)$, then we also have
that $A$ is H\"older continuous:
see~\cite[p.~298]{gilbarg-trudinger01}.  Thus, both of these results
follow from classical Schauder estimates~\cite{gilbarg-trudinger01}.  Nevertheless, since our proofs are
different from the classical ones they are of some interest.  

It remains open whether anything can be said when $p=n=2$ and $ A
\in W^{1,2}(\Omega)$ or even when $\|\nabla A\|_{L^2(\log
  L)(\Omega)}<\infty$.  We conjecture that in this endpoint case, $D^2u \in L^{2)}(\Omega)$,
where $L^{2)}$ denotes the grand Lebesgue space with norm
\[ \|f\|_{L^{2)}(\Omega)} = \sup_{0<\epsilon<1}
\left(\epsilon \avgint_\Omega
  |f(x)|^{2-\epsilon}\,dx\right)^{\frac{1}{2-\epsilon}}. \]
These spaces were introduced in~\cite{MR1176362} and have proved
useful in the study of endpoint estimates in
PDEs~\cite{MR1230483,MR1428651}.  As evidence for this conjecture, we
note that the solution $u$ given in Example~\ref{example:n=2} is in
$L^{2)}(B)$.   A stronger conjecture, also satisfied by our example, is
that $D^2u$ lies in the Orlicz space $L^2(\log L)^{-1}(\Omega)$. (This
space is a proper subset of $L^{2)}$:  see~\cite{MR1230483}.)  In both
cases our proof techniques are not sharp enough to produce these
estimates and a different approach will be required.

\medskip

The remainder of this paper is organized as follows.  In
Section~\ref{section:prelim} we state some preliminary definitions and
weighted Fefferman-Phong type inequalities that are central to our
proofs.  These results depend on recent work on two-weight norm
inequalities for the Riesz potential~\cite{MR3065302}.  In
Section~\ref{section:main-proof} we prove
Theorem~\ref{thm:main-result}.  Our proof uses ideas
from~\cite{MR1174821}.
In Section~\ref{section:n=2} we
consider the special case when $n=2$: we construct
Example~\ref{example:n=2} and sketch the proofs of
Theorems~\ref{thm:n=2-special} and~\ref{thm:n=2-off-diag}.  Throughout our notation will be
standard or defined as needed.  Given a vector matrix function, if way
say that it belongs to a scalar function space (e.g. $A\in
W^{1,n}(\Omega)$) we mean that each component function is an element
of the function space; to compute the norm we first take the $\ell^2$
norm of the components.  Constants $C$, $C(n)$, etc. may change in
value at each appearance.

% \medskip

% \noindent{\bf Acknowledgment.} The authors would like to thank Doyoon Kim for pointing out a mistake in an earlier version of this article. 

\section{Preliminary Results}
\label{section:prelim}

In this section we give conditions on a weight $w$ for the  two-weight Sobolev inequality 
$$\|fw\|_{L^p(\Omega)}\leq C\|\nabla f\|_{L^p(\Omega)} $$
to hold.   Such inequalities are sometimes
referred to as Fefferman-Phong inequalities: see~\cite{MR707957}.  
Given the classical pointwise inequality
\[|f(x)| \leq C(n)I_1(|\grad f|)(x), \qquad f \in C_0^\infty, \] 
it suffices to prove two weight estimates for the Riesz potential of
order one:
$$I_1f(x)=\Delta^{-\frac12}f(x)=c\int_{\R^n}\frac{f(y)}{|x-y|^{n-1}}\,dy.$$

In Theorems~\ref{thm:main-result} and~\ref{thm:n=2-special} we will
apply a sharp sufficient condition for the Riesz potential to be
bounded that was proved by P\'erez~\cite{perez94}; we will use the
version from~\cite[Theorem~3.6]{MR3065302} as this gives precise
values for the constants.  To state
this result, we need to make some definitions; for additional
information on Orlicz spaces and two-weight inequalities, see
~\cite{cruz-martell-perezBook,MR3065302}.  A convex, strictly
increasing function $\Phi:[0,\infty]\rightarrow [0,\infty]$ is said to
be a Young function if $\Phi(0)=0$ and $\Phi(\infty)=\infty$.  Given a
Young function there exists another Young function, $\bar{\Phi}$,
called the associate function, such that $\Phi^{-1}(t)\bar{\Phi}^{-1}(t)\simeq
t$.  For our purposes there are two particularly important examples of
Young functions that we will use. First, if $\Phi(t)=t^r$, $r>1$, then
$\bar{\Phi}(t)=t^{r'}$.  If $\Phi(t)=t^r\log(e+t)^a$, then
$\bar{\Phi}(t)\simeq t^{r'}\log(e+t)^{-\frac{a}{r-1}}$.

Given $1<p<\infty$ and a Young function $\Phi$, define
\begin{equation} \label{eqn:alpha-cond}
 \alpha_{p,\Phi} = \left(\int_1^\infty
  \frac{\Phi(t)}{t^p}\frac{dt}{t}\right)^{1/p}.
\end{equation}

Our conditions on weights  are defined using a normalized
Orlicz norm:  given Young function $\Phi$ and a cube $Q$, let
$$\|f\|_{\Phi,Q}=\inf\bigg\{\lambda>0: \avgint_Q
\Phi\left(\frac{|f(x)|}{\lambda}\right)
\,dx\leq 1\bigg\}.$$
Given a pair of weights $(u, v)$ (i.e., non-negative, locally
integrable functions) and Young functions $\Phi$ and $\Psi$, let
\[ [u,v]_{A^1_{p,\Psi,\Phi}} = 
\sup_Q |Q|^{\frac{1}{n}} \|u^{1/p}\|_{\Psi,Q}\|v^{-1/p}\|_{\Phi,Q}. \]
where the supremum is taken over all cubes in $\R^n$ with sides
parallel to the coordinate axes.  
 
\begin{theorem}{\cite[Theorem~3.6]{MR3065302}} \label{lem:fracintbd} 
Given $1<p<\infty$,  a pair
  of weights $(u,v)$, and Young functions $\Phi$ and $\Psi$,   we have that
\[ \|I_1 \|_{L^p(v)\rightarrow L^p(u)} \leq
C(n,p)[u,v]_{A^1_{p,\Psi,\Phi}} \, \alpha_{p,\bar{\Phi}}\,
\alpha_{p',\bar{\Psi}}. \]
\end{theorem}

\begin{remark}
In Theorem~\textup{\ref{lem:fracintbd}} we need to apply the integral condition
in~\eqref{eqn:alpha-cond} to the associate functions
$\bar{\Phi},\,\bar{\Psi}$.  If $\Phi$ and $\Psi$ are doubling (i.e.,
$\Phi(2t)\leq C\Phi(t)$, $t>0$, and similarly for $\Psi$), then by a
change of variables this condition can be restated in terms of $\Phi$
and $\Psi$.  See~\cite[Prop.~5.10]{cruz-martell-perezBook} for further
information.
\end{remark}

\medskip

We can now give the Sobolev inequalities needed for our results.

\begin{lemma} \label{lemma:CFlemma}
Fix $n\geq 2$ and $1<p<n$.  Let $\Omega\subset \mathbb{R}^n$.    Then, for any $f\in
W_0^{1,p}(\Omega)$ and $w\in L^n(\Omega)$, 
\begin{equation} \label{eqn:CFlemma1}
\|fw\|_{L^p(\Omega)}
\leq C(n) (p'-n')^{-1/p'}\|w\|_{L^n(\Omega)} 
\|\grad f\|_{L^p(\Omega)}.
\end{equation}
\end{lemma}

\begin{proof} 
Extend $w$ to a function on all of $\R^n$ by setting it equal to $0$
outside of $\Omega$.  Let $\Psi(t)=t^{n}$ and $\Phi(t)=t^r$, $1<r<p$; the exact value of
$r$ is not significant.   Then 
\[ \alpha_{p',\bar{\Psi}} = (p'-n')^{-1/p'}, \qquad
\alpha_{p,\bar{\Phi}} = (p-r)^{-1/p}, \]
and so we have that
\begin{multline*}
 [w^p,1]_{A^1_{p,\Psi,\Phi}} \, \alpha_{p,\bar{\Phi}}\,
\alpha_{p',\bar{\Psi}} 
\\ =
(p'-n')^{-1}(p-r)^{-1}\sup_Q |Q|^{1/n} \left(\avgint_Q
w^n\,dx\right)^{1/n} 
\leq (p'-n')^{-1}(p-r)^{-1}\|w\|_{L^n(\Omega)}.
\end{multline*}
Therefore, 
by Lemma \ref{lem:fracintbd} we have that for all $f\in C_0^\infty(\Omega)$,
\[\|fw\|_{L^p(\R^n)}\leq \|I_1(|\nabla f|)w\|_{L^p(\R^n)} \leq 
C(n,p,r)(p'-n')^{-1/p'}\|w\|_{L^n(\Omega)}\|\nabla f\|_{L^p(\R^n)}. \]
The desired inequality follows for
all $f$ by a standard approximation argument.
\end{proof}

When $n\geq 3$ we see that $w\in L^n(\Omega)$ implies the Sobolev
inequality for $p=2$.  When $n=2$ we only get the Sobolev inequality
for $1<p<2$, and the constant blows up as $p$ tends to $2$ (and also
as it tends to $1$).  In general $w\in L^2(\Omega)$ will not be a
sufficient condition for the Sobolev inequality when $p=n=2$.

To prove Theorem~\ref{thm:n=2-special} we can use the full power of
Theorem~\ref{lem:fracintbd} to prove a substitute for
Lemma~\ref{lemma:CFlemma}.  To state it, we define the non-normalized
Orlicz norm: given an open set $\Omega$ and an Orlicz function $\Psi$,
\[\|f\|_{ L^\Psi(\Omega)} = \inf\left\{ \lambda > 0 :
\int_\Omega \Psi\left(\frac{|f(x)|}{\lambda}\right)dx\leq 1
\right\}. \]
When $\Psi(t)=t^2\log(e+t)^{1+\delta}$, then we write $L^\Phi(\Omega)=
L^2(\log L)^{1+\delta}(\Omega)$. 

\begin{lemma} \label{lemma:CFlemma-n=2}
Given a bounded open set $\Omega\subset \mathbb{R}^2$ and $w\in L^2(\log
L)^{1+\delta}(\Omega)$,  if $f\in
W_0^{1,2}(\Omega)$,
then
\begin{equation} \label{eqn:CFlemma1-n=2}
\|fw\|_{L^2(\Omega)}
\leq 
C\delta^{-1/2} [1+\diam(\Omega)]\|w\|_{L^2(\log L)^{1+\delta}(\Omega)} 
\|\grad f\|_{L^2(\Omega)}.
\end{equation}
\end{lemma}

\begin{proof}
We begin as in the proof of Lemma~\ref{lemma:CFlemma}, but we now take
$\Psi(t)=t^2\log(e+t)^{1+\delta}$.    Then
\[ \alpha_{2,\bar{\Psi}} = \left(\int_1^\infty
  \frac{dt}{t\log(e+t)^{1+\delta}}\right)^{1/2}= C\delta^{-1/2}< \infty, \]
and
\[ [ w^2,1]_{A^1_{2,\Phi,\Psi}} = 
\sup_Q |Q|^{1/2}\|w\|_{\Psi,Q}. \]
Since we may assume $\supp(w)\subset \Omega$, we may restrict the
supremum to cubes $Q$ such that $|Q|\leq \diam(\Omega)^2$.  Then by
the definition of the norm, we have that
\begin{align*}
 |Q|^{1/2}\|w\|_{\Psi,Q}
&  = \inf\left\{ \lambda > 0 :
\avgint_Q \frac{|Q|w(x)^2}{\lambda^2} \log\left(e+\frac{|Q|^{1/2}
    w(x)}{\lambda}\right)^{1+\delta}dx \leq 1 \right\} \\
& \leq \inf\left\{ \lambda > 0 :
\int_\Omega \frac{w(x)^2}{\lambda^2} \log\left(e+\frac{\diam(\Omega)
    w(x)}{\lambda}\right)^{1+\delta}dx \leq 1 \right\} \\
& \leq [1+\diam(\Omega)] \|w\|_{L^\Psi(\Omega)}. 
\end{align*}
The desired inequality now follows as before. 
\end{proof}

To prove Theorem~\ref{thm:n=2-off-diag} we need an ``off-diagonal''
inequality for the Riesz potential.  There is a version of
Theorem~\ref{lem:fracintbd} in this case, but we will use a stronger
result due to D.~R. Adams~\cite{MR0336316} (see
also~\cite[Theorem~4.7.2]{MR1014685}).

\begin{theorem} \label{thm:adams}
Given $1<p<n$, $p<q<\infty$ and a weight $u$, if
\[ u(Q) \leq K|Q|^{\frac{a}{n}}, \]
where $a= \frac{q(n-p)}{p}$, then 
\[ \|I_1 f \|_{L^q(u)} \leq C(p,q,n) K^{1/q}\|f\|_{L^p}. \]
\end{theorem}

\begin{lemma} \label{lemma:off-diagonal}
Given an open set $\Omega\subset \mathbb{R}^2$,  suppose that for $1<r<2$,
$w\in  L^{2,\frac{4}{r}-2}(\Omega)$
   If $f\in
W_0^{1,2}(\Omega)$, then 
\begin{equation} \label{eqn:adams}
 \|fw\|_{L^2(\Omega)}
\leq C(r)\|w\|_{L^{2,\frac{4}{r}-2}(\Omega)} 
\|\grad f\|_{L^r(\Omega)}.
\end{equation}
\end{lemma}

\begin{proof}
Extend $w$ to a function on all of $\R^2$ by setting it equal to $0$
outside of $\Omega$.   Define $a$ as in Theorem~\ref{thm:adams} with
$p=r$, $q=n=2$.  Then for all cubes $Q$ that intersect $\Omega$,
\[ |Q|^{-\frac{a}{2}}\int_Q w(x)^2\,dx =
|Q|^{1-\frac{2}{r}}\int_{Q\cap \Omega} w(x)^2\,dx \leq
\|w\|_{L^{2,\frac{4}{r}-2}(\Omega)}^2. \]
Inequality~\eqref{eqn:adams} now follows from 
Theorem~\ref{thm:adams} and an approximation argument.
\end{proof}

\section{Proof of Theorem~\ref{thm:main-result}}
\label{section:main-proof}

We begin with two results due to Meyers \cite{MR0159110}.  The first is a coercivity condition.

\begin{theorem} \label{thm:coercive}
Given a bounded open set $\Omega\subset \R^n$ with $C^1$ boundary, let
$A$ be an $n\times n$ real-valued matrix that satisfies~\eqref{E1}.  
Define the sesquilinear form
\[ \mathfrak{a}(u,v) = \int_\Omega A\grad u \cdot \grad v\,dx. \]
Then there exists $p_0=p_0(n,\lambda,\Lambda,\Omega)$,  $1<p_0<2$, such that for all $p$, $p_0<p\leq2$, 
and all $u \in W^{1,p}_0(\Omega)$, 
\begin{equation} \label{eqn:coercive}
 \|u \|_{W^{1,p}_0(\Omega)} \approx
\sup_{\|v\|_{W^{1,p'}_0(\Omega)}=1} |\mathfrak{a}(u,v)|.
\end{equation}
Moreover, the constants in this equivalence depend on $\lambda$,
$\Lambda$, $p$, $n$, and $\Omega$.  They are independent of the specific matrix $A$.  
\end{theorem}

\begin{proof}
The upper
estimate for $\mathfrak{a}(u,v)$ is just H\"older's inequality; it is
the lower estimate that is non-trivial.  From the proof of Theorem 1
in \cite{MR0159110} we have the existence of $p_0<2$ and
$\kappa=\kappa(\lambda,\Lambda,p,\Omega)>0$ such that 
$$\inf_{\|u\|_{W^{1,p}_0}=1}\sup_{\|v\|_{W^{1,p'}_0}=1}|\mathfrak{a}(u,v)|\geq
\kappa.$$
A key hypothesis in the proof is the existence of $q>p_0'$ such that
for every $F\in L^q(\Omega)$, there exists a unique weak solution $\Phi$ to
the equation $\Delta \Phi=\Div F$ on $\Omega$ and the 
estimate
\begin{equation}\label{eqn:lapgradest} \|\nabla \Phi\|_{q}\leq C\|F\|_{q},\end{equation}
holds.  Since $\partial \Omega$ is $C^1$, by Auscher and
Qafsaoui~\cite{MR1911202}, we have that such a solution exists and \eqref{eqn:lapgradest} holds for all $q$, $1<q<\infty$.  
\end{proof}

\begin{remark}
  The value of $p_0$ is difficult to estimate from Meyer's proof.
  There is an elegant proof of this result in~\cite{MR1645727} that
  uses the Hodge decomposition.  In~\cite{MR1288682} a careful
  estimate is given for the resulting constants; though again the
  exact value is not easy to determine.  
  In passing we note that in~\cite{MR1645727},
  Theorem~\ref{thm:coercive} is proved for ``regular'' domains, which
  are defined abstractly in~\cite{MR1288682}.  However, regular
  domains include Lipschitz domains: see~\cite{MR1428651}.
\end{remark}

For our existence results we also need the following result of Meyers \cite[Theorem~1]{MR0159110}.

\begin{theorem} \label{thm:Meyers} Let $\Omega \subset \R^n$,
  $n\geq 2$, be a bounded open set such that $\partial \Omega$ is
  $C^1$, and 
let  $A$ be an $n\times n$ real-valued matrix that
satisfies~\eqref{E1}.  Then the equation
$$Lu=\Div( A\nabla u) =f$$
has a unique solution in $W_0^{1,p}(\Omega)$ for every $f\in
L^p(\Omega)$, provided $p_0<p<p_0'$, where $p_0$ is the constant from Theorem~\textup{\ref{thm:coercive}}.  The solution satisfies the estimate
$$\|\nabla u\|_{L^p(\Omega)}\leq C\|f\|_{L^p(\Omega)},$$
where $C=C(n,\lambda,\Lambda,p,\Omega)$.
\end{theorem}

\medskip

\begin{proof}[Proof of Theorem \textup{\ref{thm:main-result}}]
Fix a matrix $A$ satisfying~\eqref{E1}, and fix $p$, $p_0<p<2$, where
$p_0$ is as in Theorem~\ref{thm:coercive}.  By Theorem \ref{thm:Meyers}, for
any $f\in L^p(\Omega)$ the equation $Lu=f$ has a unique solution $u
\in W^{1,p}_0(\Omega)$ such that
\begin{equation} \label{eqn:grad-ineq}
 \|\grad u \|_{L^p(\Omega)} \leq C \|f \|_{L^p(\Omega)},
\end{equation}
with $C$ independent of $f$.

\medskip

We first prove the desired estimate on $D^2 u$ in the special case when $f\in
C^\infty(\Omega)$ and $A\in C^\infty(\Omega)$;
afterwards we will prove the general case by a double approximation
argument.  Let $u$ be the solution of~\eqref{eqn:main1}.   Then $u\in
C^\infty(\Omega)$: see Evans~\cite[Th.~3, Sec.~6.3.1]{MR2597943}.  (Note
that in this result there is an implicit assumption on the regularity of the boundary because of an
appeal to a Poincar\'e-Sobolev type inequality for functions without compact support in $\Omega$; $C^1$ is more than sufficient for this purpose.)
We now have the pointwise identity
\[ f = - \Div A\grad u = -\sum_{i,j} \big( a_{ij} u_{x_j}\big)_{x_i}.\]

Fix $s$ with $1\leq s \leq n$ and $\eta\in C^\infty_0(\Omega)$ with
$0\leq \eta\leq 1$.  Then $\eta u_{x_s}\in W_0^{1,p}(\Omega)$, so by
Theorem~\ref{thm:coercive} there exists $v\in C_0^2(\Omega)$,
$\|v\|_{W^{1,p'}_0}=1$, and
$\kappa=\kappa(n,\lambda,\Lambda,\Omega)>0$ such that
\begin{multline} \label{eqn:use-lemma}
 |\mathfrak{a}(\eta u_{x_s},v)| 
\geq \kappa \|\eta u_{x_s}\|_{W^{1,p}_0}
\geq \kappa \| \nabla\big(\eta u_{x_s}\big)\|_{L^p(\Omega)} \\
\geq \kappa \| \eta\nabla u_{x_s}\|_{L^p(\Omega)} 
- \kappa \|u_{x_s} \grad \eta\|_{L^p(\Omega)}.
\end{multline}
If we multiply $f$ by $\eta v_{x_s}$, integrate over $\Omega$ and integrate by parts twice we get
\begin{align*}
\int_\Omega f\eta v_{x_s}\,dx 
&= - \int_\Omega \sum_{i,j} \big(a_{ij} u_{x_j}\big)_{x_i} \eta
v_{x_s}\,dx \\
& =  \int_\Omega \sum_{i,j} a_{ij} u_{x_j} \big( \eta v_{x_s}\big)_{x_i}\,dx\nonumber\\
&= -\int_\Omega \sum_{i,j} \big(\eta a_{ij} u_{x_j}\big)_{x_s} v_{x_i}\,dx + \int_\Omega v_{x_s} A\nabla u\cdot \nabla \eta \, dx\nonumber\\
 &= - \int_\Omega \sum_{i,j} \eta\big(a_{ij} \big)_{x_s}u_{x_j}
v_{x_i}\,dx - \int_\Omega \eta_{x_s}  A\nabla u\cdot \nabla v\,dx\\
& \quad - \int_\Omega  \eta  A \nabla \big( u_{x_s}\big)\cdot \nabla v\,dx + \int_\Omega v_{x_s} A\nabla u\cdot \nabla \eta\,dx.\nonumber
\end{align*}
Therefore, if we take absolute values, rearrange terms, and combine
this with inequality~\eqref{eqn:use-lemma}, we get
\begin{align*}
\kappa \|\eta \grad ( u_{x_s})\|_{L^p(\Omega)} 
&\leq |\mathfrak{a}(\eta u_{x_s},v)| + \kappa \|u_{x_s}\nabla \eta\|_{L^p(\Omega)}\\
 &\leq  \left|\int_\Omega \eta A\grad(u_{x_s})\cdot \grad v\,dx\right|
 + \left|\int_\Omega u_{x_s} A\grad \eta\cdot \grad v\,dx\right|+ \kappa \|u_{x_s}\nabla \eta\|_{L^p(\Omega)}\\
 &\leq \int_\Omega \bigg|\sum_{i,j} \eta\big(a_{ij}\big)_{x_s}
  u_{x_j}v_{x_i}\bigg|\,dx +\int_\Omega |\eta_{x_s}|\,|A\nabla u\cdot\nabla v|\; dx\\
& \quad +\, \int_\Omega |v_{x_s}|\,|A\nabla u \cdot\nabla \eta|\; dx + \int_\Omega |u_{x_s}|\,|A\nabla\eta\cdot\nabla v|\; dx\nonumber\\
& \quad  + \int_\Omega |f\eta v_{x_s}|\,dx + \kappa \|u_{x_s}\nabla
\eta\|_{L^p(\Omega)} \\
& = I_1 + I_2 + I_3 + I_4 + I_5 + I_6. 
\end{align*}

We estimate each separately.  The bound for $I_5$ is straightforward:  by H\"older's inequality,
\[ I_5 \leq \|f\|_{L^p(\Omega)}\|v_{x_s}\|_{L^{p'}(\Omega)}
\leq  \|f\|_{L^p(\Omega)}\|v\|_{W_0^{1,p'}(\Omega)}
= \|f\|_{L^p(\Omega)}. \]
Similarly,  using H\"older's inequality together with \eqref{E1} we
see that
\[ 
I_2+I_3+I_4+I_6 \leq C(\kappa,\Lambda)(\sup_\Omega |\nabla \eta|)\|\nabla u\|_{L^p(\Omega)}.
\]

The key estimate is for $I_1$.   Define
\[ A_s = \big( \big(a_{ij}\big)_{x_s}\big), \qquad 
U = |\grad A| = \bigg(\sum_{i,j,s} \big(a_{ij})_{x_s}^2\bigg)^{1/2}, \]
and fix $\epsilon>0$; the exact value of $\epsilon$ will be given below.
Since $U\in L^n(\Omega)$, there exists $K=K(\epsilon,U)$ such that 
\begin{equation} \label{eqn:define-K}
 \left(\int_{\{x : U(x)>K\}} U(x)^n\,dx\right)^{1/n} < \epsilon.
\end{equation}
Let $U_1= U\chi_{\{x : U(x)>K\}}$ and $U_2=U-U_1$.  
Then, by H\"older's inequality and Lemma~\ref{lemma:CFlemma}, here
using that $\eta\nabla u\in W^{1,p}_0(\Omega)$, we can
estimate as follows:
\begin{align*}
I_1
 & = \int_\Omega |\eta A_s \grad u \cdot \grad v|\,dx \\
&  \leq \int_\Omega U|\eta\grad u||\grad v|\,dx \\
&  \leq \left(\int_\Omega |\eta\grad u\,U|^p\,dx\right)^{1/p}
\left(\int_\Omega |\grad v|^{p'}\,dx\right)^{1/p'} \\
&  \leq \left(\int_\Omega |\eta\grad u\, U_1|^p \,dx\right)^{1/p}
+ \left(\int_\Omega |\eta\grad u\, U_2|^p \,dx\right)^{1/p}\\
&  \leq C(n)(p'-n')^{-1/p'}\epsilon \left(\int_\Omega |\nabla(\eta \grad u)|^p\,dx\right)^{1/p}
+ K(\epsilon,U)\left(\int_\Omega |\grad u|^p \,dx\right)^{1/p}\\
&  \leq C(n,p)\epsilon \left(\int_\Omega |\eta D^2 u|^p\,dx\right)^{1/p}
+C(n,p)\epsilon \left(\int_\Omega |\grad \eta \cdot \grad u|^p\,dx\right)^{1/p}
+ {K}\left(\int_\Omega |\grad u|^p \,dx\right)^{1/p} \\
&  \leq C(n,p)\epsilon \left(\int_\Omega |\eta D^2 u|^p\,dx\right)^{1/p}
+ \tilde{K}(1+\|\grad \eta \|_\infty)\left(\int_\Omega |\grad u|^p \,dx\right)^{1/p},
\end{align*}
where $\tilde{K} =\tilde{K}(n,p,\ep,K)$.

Each of the above estimates hold for all values of $s$.  Therefore, by
Minkowski's inequality, if we sum over all $s$ and combine these
estimates, we get that 
\begin{multline*}
\kappa\|\eta D^2u\|_{L^p(\Omega)} \leq \sum_s \kappa \|\eta\grad(u_{x_s})\|_{L^p(\Omega)} \\
\leq C(n,p)\epsilon\|\eta D^2 u\|_{L^p(\Omega)}+ \overline{K}(1+\|\grad \eta \|_\infty)\|\grad u\|_{L^p(\Omega)} + n\|f\|_{L^p(\Omega)},
\end{multline*}
where $\overline{K}=\overline{K}((n,p,\Lambda, \epsilon,K)$.  Since $\epsilon>0$ is arbitrary, we can fix $\epsilon=\kappa/2C(n,p)$ and
then rearrange terms to get 
\begin{multline} \label{eqn:main-est}
\|\eta D^2 u\|_{L^p(\Omega)} \\
\leq {2\overline{K}\kappa^{-1}}(1+\|\grad \eta \|_\infty)\|\grad u\|_{L^p(\Omega)}
+{2n}{\kappa^{-1}}\|f\|_{L^p(\Omega)} \leq C_0 (1+\|\grad \eta \|_\infty) \|f\|_{L^p(\Omega)},
\end{multline}
where the last inequality follows from~\eqref{eqn:grad-ineq}, and $C_0=C_0(p,n,\lambda,\Lambda,\Omega,K)$.

To complete the proof, fix $\Omega'\Subset\Omega$ and choose $\eta \in
C^\infty_0(\Omega)$ such that $\eta(x) = 1$ in $\Omega'$ and so that
$\|\grad \eta\| \approx
D^{-1}$, where $D=\text{dist}(\Omega',\partial \Omega)$.  Without loss of
generality we may assume $D>1$.  With this choice of $\eta$, 
inequality~\eqref{eqn:main-est} yields the local
$W^{2,p}(\Omega)$ estimate
\begin{equation}\label{localest}
\| D^2 u \|_{L^p(\Omega')} \leq D^{-1}C\|f\|_{L^p(\Omega)},
\end{equation}
where $C=C(n, p, \lambda,\Lambda,K)$.  Finally, if we assume that
$\partial\Omega$ is $C^2$, we can apply the argument given in
\cite[p.187]{gilbarg-trudinger01} to obtain a constant $C>0$ depending
on $K,\; p,\; n,\; \lambda$, and $\Lambda$ so that
\begin{equation}\label{globest}
\| D^2 u\|_{L^p(\Omega)} \leq C \|f\|_{L^p(\Omega)}.
\end{equation}
This
completes the proof of inequality~\eqref{eqn:main2} when $f$ and $A$
are sufficiently smooth.

\bigskip

We will now prove that inequalities~\eqref{localest}
and~\eqref{globest} hold for general $f$ and $A$ satisfying the
hypotheses.  We will only consider the latter equation as the proof of
the former is essentially the same.  

We will first show that we can take an arbitrary $f$.  Fix $f\in
L^p(\Omega)$, and fix a sequence of functions $\{f_j\}$ in
$C^\infty(\Omega)$ that converge to $f$ in $L^p(\Omega)$.    Fix $A \in
C^\infty(\Omega)$ and let $u_j\in W^{1,p}_0(\Omega)$ be the solution to $L u_j =
f_j$, and let $u\in W_0^{1,p}$ be the solution to $L u = f$.  
By inequality \eqref{eqn:grad-ineq} and the Sobolev inequality, we
have that 
\[ \|u-u_j\|_{L^p(\Omega)} \leq C\|\grad (u-u_j)\|_{L^p(\Omega)} 
\leq C\|f-f_j\|_{L^p(\Omega)}.  \]
Therefore, $u_j\rightarrow u$ in $W_0^{1,p}(\Omega)$.  

Since $f_j$ and $A$ have the requisite smoothness, we can
apply~\eqref{globest} to $u_i-u_j$ to get
\[ \|D^2(u_i-u_j)\|_{L^p(\Omega)} \leq C\|f_i-f_j\|_{L^p(\Omega)}. \]
Thus, the sequence $\{u_j\}$ is Cauchy in $W^{2,p}(\Omega)$.  For
$1\leq r,\,s\leq n$, let $v_{r,s}$ denote the limit of
$\{(u_j)_{x_r,x_s}\}$.  Then for any $\phi\in C_0^\infty(\Omega)$,
\begin{equation} \label{eqn:weak}
 \int_\Omega u_{x_s} \phi_{x_r}\,dx = 
\lim_{j\rightarrow \infty} 
\int_\Omega (u_j)_{x_s} \phi_{x_r}\,dx =
\lim_{j\rightarrow \infty} 
\int_\Omega (u_j)_{x_r,x_s} \phi\,dx =
\int_\Omega v_{r,s} \phi\,dx.
\end{equation}
Therefore,  $u\in W^{2,p}(\Omega)$ and 
$u_j\rightarrow u$ in $W^{2,p}(\Omega)$.
Inequality~\eqref{eqn:main2} for $u$ now follows
immediately.

\bigskip

Finally, we prove that we can take arbitrary $A\in W^{1,n}(\Omega)$.
Fix such an $A$, and let $\{A_j\}$ be a sequence of matrices in
$C^\infty(\Omega)$ that converges to $A$ in
$W^{1,n}(\Omega)$.   It follows at once from the standard construction
of the $A_j$ (cf.~Adams and Fournier~\cite{MR2424078}) that we may
assume that the $A_j$ are elliptic with the same ellipticity constants
as $A$.  Finally,
let $U_j=|\grad A_j|$;  then $U_j \rightarrow U=|\grad A|$ in
$L^n(\Omega)$.  By the converse to the dominated convergence theorem
(see Lieb and Loss~\cite[Th.~2.7]{MR1817225}), if we pass to a subsequence,
then we may assume that $U_j \rightarrow U$ pointwise a.e., and there
exists $g\in L^n(\Omega)$ such that $U_j(x)\leq g(x)$ a.e.  Therefore,
by the dominated convergence theorem (again passing to a subsequence)
we may assume that~\eqref{eqn:define-K} holds (with fixed
$\epsilon$) for each $U_j$ with a constant $K$ independent of $j$.

Fix $f\in L^p(\Omega)$ and let $u_j\in W_0^{1,p}(\Omega)$ be the
solution of $-\Div A_j \grad u_j = f$ and let $u\in W_0^{1,p}(\Omega)$ be the solution of
$Lu= -\Div A \grad u = f$.   Then for any $\phi\in C_0^\infty(\Omega)$, 
\[ \int_\Omega A_j\grad u_j\cdot \grad \phi\,dx 
= - \int_\Omega f\phi\,dx 
= \int_\Omega A\grad u\cdot \grad \phi\,dx.  \]
Therefore,
\[ \int_\Omega \big(A\grad u - A\grad u_j+A\grad u_j - A_j\grad
u_j\big)\grad \phi\,dx = 0, \]
and so by rearranging terms we have that
\[ |\frak{a}( u - u_j, \phi)| = 
\left|\int_\Omega A(\grad u -\grad u_j)\cdot\grad \phi\,dx\right|\
\leq \int_\Omega |(A-A_j)\grad u_j\cdot \grad \phi|\,dx. \]
By Theorem~\ref{thm:coercive} there exists $\phi$ such that
$\|\phi\|_{W_0^{1,p'}(\Omega)}=1$ and $\kappa>0$ such that
\begin{multline} \label{eqn:converge}
  \kappa \|u-u_j\|_{W_0^{1,p}} \leq 
\int_\Omega |(A-A_j)\grad u_j\cdot \grad \phi|\,dx \\
\leq \|A-A_j\|_{L^n(\Omega)}
\|\grad u_j\|_{L^{\frac{np}{n-p}}(\Omega)} 
\|\grad \phi \|_{L^{p'}(\Omega)}.
\end{multline}
The last estimate follows by H\"older's inequality, since
\[ \frac{1}{n}+ \frac{n-p}{np} +\frac{1}{p'} = 1. \]
The last term on the righthand side of~\eqref{eqn:converge} is at most
$1$.  By our choice of the $A_j$, the first term tends to $0$ as
$j\rightarrow \infty$. And by the Sobolev inequality,
\[ \|\grad u_j\|_{L^{\frac{np}{n-p}}(\Omega)} \leq 
C\|D^2 u_j\|_{L^p(\Omega)} \leq C\|f\|_{L^p(\Omega)}; \]
the final inequality holds since by our choice of the $A_j$,
inequality~\eqref{eqn:main2} holds for each  $u_j$ with a constant
independent of $j$.  
 Therefore, the
middle term on the righthand side of~\eqref{eqn:converge} is uniformly
bounded.  Hence, $u_j\rightarrow u$ in
$W_0^{1,p}(\Omega)$.

It remains to show $D^2u$ exists and estimate its norm.   By
inequality~\eqref{eqn:main2}, the sequence $\{D^2u_j\}$ is uniformly
bounded in $L^p(\Omega)$, and so has a weakly convergent subsequence.
Passing to this subsequence, we can repeat the argument
at~\eqref{eqn:weak} to conclude that $u\in W^{2,p}(\Omega)$ and
$D^2u_j$ converges weakly to $D^2u$.  But then we have that
\[  \|D^2u\|_{L^p(\Omega)} \leq \liminf_{j\rightarrow \infty}
 \|D^2u_j\|_{L^p(\Omega)} \leq C\|f\|_{L^p(\Omega)}, \]
and this completes the proof.
\end{proof}

\section{The Case $n=2$}
\label{section:n=2}

In this section we consider the two dimensional case.  We first
construct Example~\ref{example:n=2} and then prove
Theorems~\ref{thm:n=2-special} and~\ref{thm:n=2-off-diag}.  

\medskip

\begin{proof}[Construction of Example~\textup{\ref{example:n=2}}]
  Our example is adapted from one given by Clop {\em et
    al.}~\cite[p.~205]{MR2474121} and is based on the theory of
  quasiregular mappings.  Let $B=B_{1/2}(0)$ and let $z=x+iy$.  Define
$f(z)=z(1-2\log|z|)$.
Then
$$\D f(z)=-2\log|z| \quad \text{and} \quad \Dbar
f(z)=\frac{z}{\bar{z}},$$
and so $f$ satisfies the Beltrami equation 
$\Dbar f=\mu\,\D f$
with Beltrami coefficient 
$$\mu(z)=\frac{z}{\bar{z}\log(|z|^{-2})}=\frac{z^2}{|z|^2\log(|z|^{-2})}.$$

If we let let $u= \re f$, that is,
$$u(x,y)=x(1-\log(x^2+y^2)),$$
then $u$ satisfies the equation
$$-\Div(A\nabla u)=0$$
where $A$ is the symmetric, real-valued matrix
\[ A =\left[\begin{array}{cc}\displaystyle \frac{|1-\mu|^2}{1-|\mu|^2} &
    \displaystyle\frac{-2\im \mu}{1-|\mu|^2} \vspace{2mm} \\
    \displaystyle \frac{-2\im \mu}{1-|\mu|^2}&\displaystyle
    \frac{|1+\mu|^2}{1-|\mu|^2} \end{array}\right]
=\frac{1+\sigma^2}{1-\sigma^2}{\bf
  Id}-\frac{2}{1-\sigma^2}\left[\begin{array}{cc}\al & \beta \\ \beta
    & -\al\end{array}\right], 
\] 
and
\[ \sigma=|\mu|=\frac{-1}{\log(x^2+y^2)}, \quad
\al=\re \mu=\frac{x^2-y^2}{x^2+y^2}\sigma, \quad 
\beta=\im \mu=\frac{2xy}{x^2+y^2}\sigma. 
\]
This follows from a straightforward calculation:  for the details,
see~\cite[p.~412]{MR2472875}.  

We claim that $A$ is elliptic and in $W^{1,2}(B)$, and that $u\in
W^{2,p}(B)$ for $p<2$ but not when $p=2$.  By our choice of domain, $0\leq
\sigma\leq k=(\log 4)^{-1}$.  Let
$\xi=(\xi_1,\xi_2)\in \R^2$; then
\begin{equation} \label{eqn:inner}
\langle A\xi,\xi\rangle=
\frac{1+\sigma^2}{1-\sigma^2}|\xi|^2-\frac{2\al(\xi_1^2-\xi_2^2)
+4\beta\xi_1\xi_2}{1-\sigma^2}.
\end{equation}
Since
$$\al(\xi_1^2-\xi_2^2)+4\beta\xi_1\xi_2=2(\al,\beta)\cdot
(\xi_1^2-\xi_2^2,2\xi_1\xi_2),$$
by the Cauchy-Schwarz inequality we have that 
$$|2\al(\xi_1^2-\xi_2^2)+4\beta\xi_1\xi_2|
\leq 2\sqrt{\al^2+\beta^2}\sqrt{(\xi_1^2-\xi_2^2)^2+4\xi_1^2\xi_2^2}=2\sigma |\xi|^2.$$
Hence,
$$-2\sigma|\xi|^2\leq 2\al(\xi_1^2-\xi_2^2)+4\beta\xi_1\xi_2\leq
2\sigma|\xi|^2,$$
and if we combine this with inequality~\eqref{eqn:inner}, we get
$$\frac{1-k}{1+k}|\xi|^2
\leq \frac{1-\sigma}{1+\sigma}|\xi|^2
\leq \langle A\xi,\xi\rangle\leq \frac{1+\sigma}{1-\sigma}|\xi|^2
\leq \frac{1+k}{1-k}|\xi|^2.  $$
Thus, $A$ is elliptic with $\lambda=\frac{1-k}{1+k}$ and $\Lambda=\frac{1+k}{1-k}$.

To see that $A=(a_{ij})\in W^{1,2}(B)$, a lengthy (and {\em
  Mathematica} assisted) calculation shows that
\begin{multline*}
\frac{\partial a_{11}}{ \partial x}
  =\\
-\frac{4x[x^2-y^2-2y^2\log^3(x^2+y^2)+(x^2-y^2)\log^2(x^2+y^2)
+2(x^2+2y^2)\log(x^2+y^2)]}{(x^2+y^2)^2(\log^2(x^2+y^2)-1)^2}
\end{multline*}
and the derivatives $\frac{\partial}{\partial x} a_{ij}$ and $\frac{\partial}{\partial y} a_{ij}$ are similar.
It follows that
\[ \left|\frac{\partial}{\partial x} a_{ij}\right|, \left|\frac{\partial}{\partial y} a_{ij}\right| \leq
C\frac{|\log^3(x^2+y^2)|}{(x^2+y^2)^{\frac12}(\log^2(x^2+y^2)-1)^2} \in
L^2(B). \]

Finally to see that $u\in W^{2,p}(B)$ for $p<2$ but not in
$W^{2,2}(B)$, another calculation shows that 
$$u_{xx}(x,y)=\frac{-2x(x^2+3y^2)}{(x^2+y^2)^2}, 
\ \ u_{xy}(x,y)=\frac{-2y(y^2-x^2)}{(x^2+y^2)^2}, 
\ \ u_{yy}(x,y)=\frac{-2x(x^2-y^2)}{(x^2+y^2)^2}.$$  
Thus, each second derivative is bounded by a constant multiple of
$(x^2+y^2)^{-\frac12}\in L^p(B)$, so $u\in W^{2,p}$.  On the other
hand,
$$\int_B |u_{xx}|^2\,dxdy=\infty,$$
so $u\not\in W^{2,2}(B)$.
\end{proof}

\begin{proof}[Proof of Theorem~\textup{\ref{thm:n=2-special}}]
Most of the proof is identical to the proof of
Theorem~\ref{thm:main-result}, setting $n=p=2$.  However, in two
places we need to make specific changes to the proof.   The proof for
$f$ and $A$ smooth is
the same up to inequality~\eqref{eqn:define-K}.    We again split $U$,
but now we fix $\epsilon$ (to be determined below) and find $K$ such
that 
\begin{equation} \label{eqn:define-K-n=2}
\|U\chi_{\{U>K\}}\|_{L^\Psi(\Omega)}<\epsilon,
\end{equation}
where $\Psi(t)=t^2\log(e+t)^{1+\delta}$. 
(This is again possible by the dominated convergence theorem in the
context of Orlicz spaces.)
Let $U=U_1+U_2=U\chi_{\{U>K\}}+U\chi_{\{U\leq K\}}$;  then by
Lemma~\ref{lemma:CFlemma-n=2},
\begin{multline} \label{eqn:use-FPh}
\left( \int_\Omega (|\eta\nabla u|U)^2\,dx\right)^{1/2}
\leq \left(\int_\Omega (|\eta\nabla u|U_1)^2\,dx\right)^{1/2}
+K\left(\int_\Omega |\nabla u|^2\,dx\right)^{1/2}\\
\leq \epsilon \, C(\delta,\Omega)  \left(\int_\Omega |D^2
  u|^2\,dx\right)^{1/2}
+\overline{K}(1+\|\grad \eta\|_\infty)\|f\|_{L^2(\Omega)}.  
\end{multline}
The argument now proceeds as before, yielding
$$\|\eta D^2u\|_{L^2(\Omega)}\leq C_0(1+\|\grad \eta\|_\infty)\|f\|_{L^2(\Omega)},$$
where again the constant $C_0=C_0(n,p,\lambda,\Lambda,\Omega,K)$.

\medskip

The proof for arbitrary $f\in L^2(\Omega)$ goes through exactly as
before.  For the proof for arbitrary $\grad A \in L^\Psi(\Omega)$,
note first that by the Sobolev embedding theorem we have $A\in
L^\Psi(\Omega)$.  We now fix smooth  
 $A_j\rightarrow A$ in $W^{1,\Psi}(\Omega)$ (the Sobolev space defined
 with respect to the $L^\Psi$ norm), and we may again assume that
 the $A_j$ have the same ellipticity constants and that we may choose
 $K$ such that \eqref{eqn:define-K-n=2} holds for all $U_j=|\grad
 A_j|$ with a constant $K$ independent of $j$.  This is possible
 since all the arguments for $W^{1,p}(\Omega)$ extend to
 $W^{1,\Psi}(\Omega)$  with almost no change.  Smooth functions are
 dense (see ~\cite{MR2424078}) and the proof of density again shows that
 ellipticity constants are preserved.  The converse of dominated
 convergence also holds in this setting; the proof is implicit in the
 literature.  For a proof in a different context that readily adapts
 to Orlicz spaces, see~\cite[Prop.~2.67]{cruz-fiorenza-book}.

 The proof now continues as before until
 inequality~\eqref{eqn:converge}.  Here we need to apply the
 generalized H\"older's inequality in the scale of Orlicz spaces
 (see~\cite[Lemma~5.2]{cruz-martell-perezBook}).  If we let $\Phi(t)=
 \exp(t^{\frac{2}{1+\delta}})-1$, then
\[ \Psi^{-1}(t)\Phi^{-1}(t) \approx 
\frac{t^{1/2}}{\log(e+t)^{\frac{1+\delta}{2}}}
\log(e+t)^{\frac{1+\delta}{2}}
\lesssim t^{1/2}. \]
Therefore, we can estimate as follows: 
\begin{multline*}
\left|\int_\Omega A(\nabla u-\nabla u_j)\cdot \nabla \phi\,dx\right|
\leq \int_\Omega |(A-A_j)\nabla u_j\cdot\nabla \phi|\,dx \\
\leq  \|(A-A_j)\nabla u_j\|_{L^2(\Omega)}\|\nabla\phi\|_{L^2(\Omega)}
\leq  C\|A-A_j\|_{L^\Psi(\Omega)} \|\nabla u_j\|_{L^\Phi(\Omega)}
  \|\nabla\phi\|_{L^2(\Omega)}. 
\end{multline*}
As in the previous argument, we have chosen $\phi$ so that
$\|\nabla\phi\|_{L^2(\Omega)}\leq 1$.  We also have that $
\|A-A_j\|_{L^\Psi(\Omega)}\rightarrow 0$ as $j\rightarrow \infty$.
Therefore, we could complete the proof as before if we can show that 
\[ \|\nabla u_j\|_{L^\Phi(\Omega)} \leq C\|f\|_{L^2(\Omega)} \]
with a constant independent of $j$.  

Let $\Phi_0(t) = \exp(t^2)-1$.  Then for $t\geq 1$, $\Phi(t)\leq
\Phi_0(t)$, and so by the properties of Orlicz norms
(see~\cite[Sec.5.2]{cruz-martell-perezBook}) there exists a constant
depending on $\delta$ and $\Omega$
such that $\|\nabla u_j\|_{L^\Phi(\Omega)}\leq C \|\nabla
u_j\|_{L^{\Phi_0}(\Omega)}$.   But by Trudinger's
inequality~\cite[Thm.~2.9.1]{MR1014685} we have the endpoint Sobolev
inequality:
\[ \|\nabla u_j\|_{L^{\Phi_0}(\Omega)} \leq C\|D^2u_j\|_{L^2(\Omega)}. \]
By the first part of the proof we have that
$\|D^2u_j\|_{L^2(\Omega)}\leq C\|f\|_{L^2(\Omega)}$ with a constant
independent of $j$; combining these inequalities we get the desired
estimate and this completes the proof.
\end{proof}

\medskip

\begin{proof}[Proof of Theorem~\textup{\ref{thm:n=2-off-diag}}]
The proof is nearly identical to the proof of
Theorem~\ref{thm:n=2-special}.  Let $\Psi(t)=t^2\log(e+t)$.  The first
half of the proof for smooth $f$ and $A$ is the same until
\eqref{eqn:use-FPh}.  Here, we use the off-diagonal estimate in Lemma~\ref{lemma:off-diagonal} and
H\"older's inequality to get
\begin{align*}
& \left( \int_\Omega (|\eta\nabla u|U)^2\,dx\right)^{1/2} \\
& \qquad \qquad \leq \left(\int_\Omega (|\eta \nabla u|U_1)^2\,dx\right)^{1/2}
+K\left(\int_\Omega |\nabla u|^2\,dx\right)^{1/2}\\
& \qquad \qquad \leq \epsilon \, C(\delta,\Omega)  \left(\int_\Omega |\eta D^2
  u|^r\,dx\right)^{1/r}
+\overline{K}(1+\|\grad \eta\|_\infty)\|f\|_{L^2(\Omega)} \\
& \qquad \qquad \leq \epsilon \, C(\delta,\Omega)
|\Omega|^{\frac{1}{(2/r)'}}\left(\int_\Omega |\eta D^2
  u|^2\,dx\right)^{1/2}
+\overline{K}(1+\|\grad \eta\|_\infty)\|f\|_{L^2(\Omega)}.  
\end{align*}
We can now complete the proof of the smooth case as before.  

The remainder of the proof goes through before, only now we apply the
generalized H\"older's inequality with $\Psi(t)$ and
$\Phi(t)=\exp(t^2)-1$ and then directly apply Trudinger's inequality.
\end{proof}

\begin{remark}
Note that in the proof of Theorem~\textup{\ref{thm:n=2-off-diag}} we use
the regularity assumption on $\grad A$ in the proof of the smooth
case, and use the higher integrability assumption on $A$ in the density
argument to prove the general case.
\end{remark}

\bibliographystyle{plain}
\bibliography{strong-deriv}

\end{document}